\documentclass{amsart}

\usepackage{amssymb}
\newtheorem{theorem}{Theorem}[section]
\newtheorem{lemma}[theorem]{Lemma}
\newtheorem{proposition}[theorem]{Proposition}

\theoremstyle{definition}
\newtheorem{definition}[theorem]{Definition}
\newtheorem{example}[theorem]{Example}

\theoremstyle{remark}
\newtheorem{remark}[theorem]{Remark}

\allowdisplaybreaks

\begin{document}

\title{Rota-Baxter Coalgebras}

\author{Run-Qiang Jian}
\address{School of Computer Science, Dongguan University
of Technology, 1, Daxue Road, Songshan Lake, 523808, Dongguan, P.
R. China}
\email{jianrq@dgut.edu.cn}

\author{Jiao Zhang${}^{*}$}
\address{Department of Mathematics, Shanghai University, 99 Shangda Road, BaoShan District, 200444, Shanghai, P.R. China.}
\email{zhangjiao@shu.edu.cn}
\thanks{${}^{*}$Corresponding author. E-mail address: zhangjiao@shu.edu.cn}


\date{}

\dedicatory{}


\begin{abstract}
We introduce the notion of Rota-Baxter coalgebra which can be viewed as the dual notion of Rota-Baxter algebra. We provide some concrete examples and establish various properties of this new object. We also consider comodules over Rota-Baxter coalgebras.
\end{abstract}

\maketitle

\section{Introduction}

A Rota-Baxter algebra of weight $\lambda$ is a pair $(R,P)$
consisting of an associative algebra $R$ and an endomorphism $P$
of $R$ such that $$P(x)P(y)=P(xP(y))+P(P(x)y)+\lambda P(xy),\ \
\forall x,y\in R.$$ They are first initiated by the work of G.
Baxter on probability theory (\cite{Bax}) and formulated formally
by Rota (\cite{Rot}) in 1960s. Later, even though some famous
mathematicians, such as Cartier, studied these new subjects, but
they did not draw people's attention extensively. Things are
changed in 2000 after the works \cite{G} and \cite{GK}. From then
on, these algebras become popular. Nowadays, except for their own
interests (see e.g., \cite{CG}), they have many applications in
other areas of mathematics and mathematical physics, such as
combinatorics (\cite{EGP}), Loday type algebras (\cite{E},
\cite{EG2}), pre-Lie and pre-Poisson algebras (\cite{LHB},
\cite{AB}, \cite{A}), multiple zeta values (\cite{EG3},
\cite{GZ}), and Connes-Kreimer renormalization theory in quantum
field theory (\cite{CK1}, \cite{CK2}, \cite{EGK1}, \cite{EGK2}). For a detailed
description of the theory of Rota-Baxter algebras, we refer the
reader to the clearly written book \cite{Guo}.

Stimulating by the theory of coalgebras, it seems reasonable to
consider the dual theory of Rota-Baxter algebras. The significance
of studying dual theories is that they do not only bring us to new
phenomena but also help us to understand the original subjects
much better. The purpose of this paper is intended to investigate
the dual theory of Rota-Baxter algebras. We first define the
notion of Rota-Baxter coalgebra. It can be viewed as the dual
version of Rota-Baxter algebras. An active mathematical subject should have certainly abundant examples. So we provide various examples of
Rota-Baxter coalgebras, including constructions by group-like
elements, by wedge products, and by smash coproducts. It is a direct and easy
verification that the dual coalgebra of a Rota-Baxter algebra
together with the linear dual of the Rota-Baxter operator is a
Rota-Baxter coalgebra. But the converse is not true in general.
One of the troubles is that the linear dual of an algebra is not a
coalgebra unless it is finite dimensional. In order to overcome
this difficulty, we restrict our attention to the finite dual of
an algebra. To make the converse statement hold, we also have to
impose some relevant conditions on the original algebra structure
that relate to the so-called double product. We study the relation
between Rota-Baxter coalgebras with counit and those without
counit in the idempotent case as well. On the other hand, we want
to establish the properties of Rota-Baxter coalgebras having dual
versions in Rota-Baxter algebras. For example, we construct the
double coproduct which is dual to the double product, and we
provide necessary and sufficient conditions for a Rota-Baxter
coalgebra being idempotent. We provide a dual version of Atkinson's additive decomposition as well. Finally, inspired by the work \cite{GL}, we introduce the notion of Rota-Baxter comodules over Rota-Baxter coalgebras. We establish the relations between these comodules and Rota-Baxter modules over Rota-Baxter algebras.

This paper is organized as follows. In Section 2, we introduce the
notion of Rota-Baxter coalgebra and provide several concrete
examples. In Section 3, we establish various properties of
Rota-Baxter coalgebras. In the last section, we define and study Rota-Baxter comodules over Rota-Baxter coalgebras.

\section{Rota-Baxter coalgebras}

Throughout this paper, we fix a ground field $\mathbb{K}$ of
characteristic 0 and assume that all vector spaces, algebras,
coalgebras and tensor products are defined over $\mathbb{K}$.

We adopt Sweedler's notation for coalgebras and comodules. Let
$(C,\Delta,\varepsilon)$ be a coalgebra. For any $c\in C$, we
denote
$$\Delta(c)=\sum c_{(1)}\otimes  c_{(2)}.$$ Let $(M,\delta)$ be a left $C$-comodule. For any $ m\in
M$, we denote
$$\delta(m)=\sum m_{(-1)}\otimes  m_{(0)}.$$

\begin{definition}A \emph{Rota-Baxter coalgebra} of weight $\lambda$ is a coalgebra $(C,\Delta, \varepsilon)$ equipped with an operator $P$ such that
\begin{equation}
  \label{RBO1}(P\otimes P) \Delta=(\mathrm{id}\otimes P)  \Delta   P+(P\otimes \mathrm{id})  \Delta   P+\lambda \Delta  P.
\end{equation}
In Sweedler's notation, this means
\begin{align*}
  &\sum P(c_{(1)})\otimes P(c_{(2)})\\
  &=\sum P(c)_{(1)}\otimes P(P(c)_{(2)})+\sum P(P(c)_{(1)})\otimes
P(c)_{(2)}+\lambda\sum P(c)_{(1)}\otimes P(c)_{(2)}.
\end{align*}
We always use the quadruple $(C,\Delta, \varepsilon,P)$ to denote a Rota-Baxter coalgebra. If $(C,\Delta)$ is a noncounitary coalgebra equipped with  an endomorphism $P$ which satisfies equation \eqref{RBO1} above, we call the triple $(C,\Delta, P)$ a \emph{noncounitary Rota-Baxter coalgebra}.

A Rota-Baxter coalgebra is called \emph{idempotent} if the
Rota-Baxter operator $P$ satisfies $P^2=P$.

A subcoalgebra $D$ of $C$ is called a \emph{Rota-Baxter subcoalgebra} of $C$ if $P(D)\subset D$.\end{definition}

\begin{remark}(1) If $(C,\Delta, P)$ is a (noncounitary) Rota-Baxter coalgebra of weight $\lambda\neq 0$, then it is a direct verification that $(C,\Delta, \mu \lambda^{-1}P)$ is a (noncounitary) Rota-Baxter coalgebra of weight $\mu$.

(2) If $(C,\Delta, \varepsilon,P)$ is a Rota-Baxter coalgebra,
then its dual $(C^\ast,\Delta^\ast, \varepsilon^\ast,P^\ast)$ is a
unitary Rota-Baxter algebra. Similarly, the dual of a noncounitary
Rota-Baxter coalgebra is a nonunitary Rota-Baxter algebra.

(3) Every Rota-Baxter subcoalgebra of a Rota-Baxter coalgebra is itself a Rota-Baxter coalgebra.
\end{remark}

\begin{proposition}\label{Idempotent RB coalgebras}Let $(C,\Delta,\varepsilon)$ be a coalgebra. Assume that $P:C\rightarrow C$ is a linear map such that $\Delta  P=(P\otimes P)  \Delta $ and $P^2=P$. Then $(C,\Delta,\varepsilon,P)$ is an idempotent Rota-Baxter coalgebra of weight $-1$.\end{proposition}\begin{proof} Indeed, for any $c\in C$, we have\begin{align*}
\lefteqn{(P\otimes \mathrm{id})  \Delta  P(c)+(\mathrm{id}\otimes P)  \Delta  P(c)-\Delta  P(c)}\\[3pt]
&=\sum P^2(c_{(1)})\otimes P(c_{(2)})+\sum  P(c_{(1)})\otimes P^2(c_{(2)})-\sum P(c_{(1)})\otimes P(c_{(2)}) \\[3pt]
&=\sum P(c_{(1)})\otimes P(c_{(2)})+\sum P(c_{(1)})\otimes P(c_{(2)})-\sum P(c_{(1)})\otimes P(c_{(2)})\\[3pt]
&=\sum P(c_{(1)})\otimes P(c_{(2)})\\[3pt]
&=(P\otimes P)  \Delta (c).
\end{align*}\end{proof}

\begin{example}Let $(C,\Delta,\varepsilon)$ be a coalgebra. Given a group-like element $g\in C$, we define $P_g:C\rightarrow
C$ by $P_g(c)=\varepsilon(c)g$ for any $c\in C$. Since
$\Delta(g)=g\otimes g$ and $\varepsilon (g)=1$, we
have\begin{align*}
 \Delta  P_g(c)&=\Delta(\varepsilon(c)g)\\[3pt]
&=\varepsilon(c)g \otimes g\\[3pt]
&=\varepsilon(\sum \varepsilon (c_{(1)})c_{(2)})g \otimes g\\[3pt]
&=\sum \varepsilon(c_{(1)})g \otimes \varepsilon (c_{(2)})g\\[3pt]
&=\sum P_g(c_{(1)}) \otimes P_g(c_{(2)}) \\[3pt]
&=(P_g\otimes P_g)  \Delta (c),
\end{align*}and
\begin{align*}
P_g^2(c)&=P_g(\varepsilon(c)g)\\[3pt]
&=\varepsilon(c)\varepsilon(g)g\\[3pt]
&=\varepsilon(c)g\\[3pt]
&=P_g(c).
\end{align*}Hence $(C,\Delta,\varepsilon,P_g)$ is an idempotent Rota-Baxter coalgebra of weight $-1$.
\end{example}

The condition $\Delta  P=(P\otimes P)  \Delta $ says that $P$ is a morphism of coalgebra. Here we do not require the compatibility of $P$ and the counit. This condition is natural and useful.

Let $(C,\Delta,\varepsilon)$ be a coalgebra, and $X$ and $Y$ be two subcoalgebras. The wedge product $X\wedge Y$ of $X$ and $Y$ is defined to be $$X\wedge Y=\Delta^{-1}(X\otimes C+C\otimes Y).$$ It is a subcoalgebra of $C$ (see Proposition 9.0.0 in \cite{Sw}).

\begin{proposition}Let $(C,\Delta,\varepsilon,P)$ be a Rota-Baxter coalgebra such that $\Delta  P=(P\otimes P)  \Delta $. For any two Rota-Baxter subcoalgebras $X$ and $Y$, their wedge product $X\wedge Y$ is again a Rota-Baxter subcoalgebra of $C$.\end{proposition}
\begin{proof}The only thing need to prove is that $P(X\wedge Y)\subset X\wedge Y$. For any $c\in X\wedge Y$, we have $\Delta(c)\subset X\otimes C+C\otimes Y$. It follows from $\Delta  P=(P\otimes P)  \Delta $ that $$\Delta  P(c)=(P\otimes P)  \Delta(c)\subset  P(X)\otimes P(C)+P(C)\otimes P(Y)\subset  X\otimes C+C\otimes Y ,$$as desired.\end{proof}

Let $(H,\Delta,\varepsilon,S)$ be a Hopf algebra. Define the left $H$-comodule structure $\delta$ on $H$ by $\delta(h)=\sum h_{(1)}S(h_{(3)})\otimes h_{(2)}$. Then $H$ is a comodule-coalgebra. So we have the smash coalgebra strucrure on $H\otimes H$ given by $$\Delta_s(a\otimes b)=\sum a_{(1)}\otimes a_{(2)}S(a_{(4)})b_{(1)}\otimes a_{(3)}\otimes b_{(2)},$$ and $$\varepsilon_s(a\otimes b)=\varepsilon(a)\varepsilon(b).$$ We define endomorphisms $P_1$ and $P_2$ of $H\otimes H$ respectively by $$P_1(a\otimes b)=a\otimes \varepsilon(b)1_H,$$ and $$P_2(a\otimes b)=\sum
S(a_{(2)}S(a_{(4)})b_{(2)})a_{(3)}\otimes
S(a_{(1)}S(a_{(5)})b_{(1)})b_{(3)}.$$ Then we have

\begin{proposition}The quadruples $(H\otimes H, \Delta_s,\varepsilon_s,P_1)$ and $(H\otimes H, \Delta_s,\varepsilon_s,P_1)$ are idempotent Rota-Baxter coalgebras of weight $-1$. \end{proposition}
\begin{proof}Indeed, we have
\begin{align*}
\lefteqn{(\mathrm{id}\otimes P_1)  \Delta_s   P_1(a\otimes
b)+(P_1\otimes
\mathrm{id})  \Delta_s   P_1(a\otimes b)- \Delta_s  P_1(a\otimes b)}\\[3pt]
&=(\mathrm{id}\otimes P_1)  \Delta_s   (a\otimes\varepsilon(b)
1_H)+(P_1\otimes
\mathrm{id})  \Delta_s   (a\otimes\varepsilon(b) 1_H)- \Delta_s  (a\otimes\varepsilon(b) 1_H)\\[3pt]
&=\varepsilon(b)\sum a_{(1)}\otimes a_{(2)}S(a_{(4)})\otimes P_1(a_{(3)}\otimes 1_H)\\[3pt]
&\ \ \ +\varepsilon(b)\sum P_1(a_{(1)}\otimes a_{(2)}S(a_{(4)}))\otimes a_{(3)}\otimes 1_H\\[3pt]
&\ \ \ -\varepsilon(b)\sum a_{(1)}\otimes a_{(2)}S(a_{(4)})\otimes a_{(3)}\otimes 1_H\\[3pt]
&=\varepsilon(b)\sum a_{(1)}\otimes a_{(2)}S(a_{(4)})\otimes a_{(3)}\otimes 1_H\\[3pt]
&\ \ \ +\varepsilon(b)\sum a_{(1)}\otimes \varepsilon(a_{(2)})\varepsilon(a_{(4)})1_H\otimes a_{(3)}\otimes 1_H\\[3pt]
&\ \ \ -\varepsilon(b)\sum a_{(1)}\otimes a_{(2)}S(a_{(4)})\otimes a_{(3)}\otimes 1_H\\[3pt]
&=\varepsilon(b)\sum a_{(1)}\otimes \varepsilon(a_{(2)})\varepsilon(a_{(4)})1_H\otimes a_{(3)}\otimes 1_H\\[3pt]
&=\sum a_{(1)}\otimes \varepsilon(a_{(2)})\varepsilon(a_{(4)})\varepsilon(b_{(1)})1_H\otimes a_{(3)}\otimes \varepsilon(b_{(2)})1_H\\[3pt]
&=\sum P_1(a_{(1)}\otimes a_{(2)}S(a_{(4)})b_{(1)})\otimes P_1(a_{(3)}\otimes b_{(2)})\\[3pt]
&=(P_1\otimes P_1) \Delta_s(a\otimes b).
\end{align*} Obviously, $P_1^2=P_1$.

On the other hand, the space $H\otimes H$ is a left $H$-module via the diagonal
action, i.e., $h.(a\otimes b)=\sum h_{(1)}a\otimes h_{(2)}b$, for
any $a,b,h\in H$. Then the action of $P_2$ can be rewritten as
$$P_2(a\otimes b)=\sum S(a_{(1)}S(a_{(3)})b_{(1)}).(a_{(2)}\otimes
b_{(2)}).$$

We first show that
\begin{align*}
P_2(h.(a\otimes b))&=\varepsilon(h)P_2(a\otimes b),\\[3pt]
\Delta_s(h.(a\otimes b))&=h. \Delta_s(a\otimes b).
\end{align*}
Notice that
\begin{align*}
P_2(h.(a\otimes b))&=\sum P_2(h_{(1)}a\otimes h_{(2)}b)\\ &=\sum S(h_{(1)}a_{(1)}S(h_{(3)}a_{(3)})h_{(4)}b_{(1)}).(h_{(2)}a_{(2)}\otimes h_{(5)}b_{(2)})\\[3pt]
&=\sum S(h_{(1)}a_{(1)}S(a_{(3)})b_{(1)}).(h_{(2)}a_{(2)}\otimes h_{(3)}b_{(2)})\\[3pt]
&=\sum (S(h_{(1)}a_{(1)}S(a_{(3)})b_{(1)})h_{(2)}).(a_{(2)}\otimes b_{(2)})\\[3pt]
&=\sum \varepsilon(h)S(a_{(1)}S(a_{(3)})b_{(1)}).(a_{(2)}\otimes b_{(2)})\\[3pt]
&=\varepsilon(h)P_2(a\otimes b).
\end{align*} So we get the first equation.
For the second one,
\begin{align*}
\Delta_s(h.(a\otimes b))&= \Delta_s(\sum h_{(1)}a\otimes h_{(2)}b)\\[3pt]
&=\sum h_{(1)}a_{(1)}\otimes h_{(2)}a_{(2)}S(h_{(4)}a_{(4)})h_{(5)}b_{(1)}\otimes h_{(3)}a_{(3)}\otimes h_{(6)}b_{(2)}\\[3pt]
&=\sum h_{(1)}a_{(1)}\otimes h_{(2)}a_{(2)}S(a_{(4)})b_{(1)}\otimes h_{(3)}a_{(3)}\otimes h_{(4)}b_{(2)}\\[3pt]
&=h. \sum a_{(1)}\otimes a_{(2)}S(a_{(4)})b_{(1)}\otimes a_{(3)}\otimes b_{(2)}\\[3pt]
&=h.\Delta_s(a\otimes b).
\end{align*}
Then we have
\begin{align*}
(P_2\otimes P_2)\Delta_s(a\otimes b)&=\sum P_2(a_{(1)}\otimes a_{(2)}S(a_{(4)})b_{(1)})\otimes P_2(a_{(3)}\otimes b_{(2)})\\[3pt]
&=\sum P_2(a_{(1)}.(1_H\otimes S(a_{(3)})b_{(1)}))\otimes P_2(a_{(2)}\otimes b_{(2)})\\[3pt]
&=\sum \varepsilon(a_{(1)})P_2(1_H\otimes S(a_{(3)})b_{(1)})\otimes P_2(a_{(2)}\otimes b_{(2)})\\[3pt]
&=\sum P_2(1_H\otimes S(a_{(2)})b_{(1)})\otimes P_2(a_{(1)}\otimes b_{(2)})\\[3pt]
&=\sum S(S(a_{(3)})b_{(1)}).(1_H\otimes S(a_{(2)})b_{(2)})\otimes
P_2(a_{(1)}\otimes b_{(3)}),
\end{align*}
and \begin{align*}
&\Delta_sP_2(a\otimes b)\\[3pt]
&=\Delta_s(\sum S(a_{(1)}S(a_{(3)})b_{(1)}).(a_{(2)}\otimes b_{(2)}))\\[3pt]
&=\sum S(a_{(1)}S(a_{(3)})b_{(1)}).\Delta_s (a_{(2)}\otimes b_{(2)}))\\[3pt]
&=\sum S(a_{(1)}S(a_{(6)})b_{(1)}).(a_{(2)}\otimes a_{(3)}S(a_{(5)})b_{(2)}\otimes a_{(4)}\otimes b_{(3)})\\[3pt]
&=\sum S(a_{(2)}S(a_{(7)})b_{(2)}).(a_{(3)}\otimes a_{(4)}S(a_{(6)})b_{(3)})\otimes S(a_{(1)}S(a_{(8)})b_{(1)}).(a_{(5)}\otimes b_{(4)})\\[3pt]
&=\sum (S(a_{(2)}S(a_{(6)})b_{(2)})a_{(3)}).(1_H\otimes S(a_{(5)})b_{(3)})\otimes S(a_{(1)}S(a_{(7)})b_{(1)}).(a_{(4)}\otimes b_{(4)})\\[3pt]
&=\sum S(S(a_{(4)})b_{(2)}).(1_H\otimes S(a_{(3)})b_{(3)})\otimes
S(a_{(1)}S(a_{(5)})b_{(1)}).(a_{(2)}\otimes b_{(4)}).
\end{align*}So
\begin{align*}
&( P_2\otimes\mathrm{id})\Delta_sP_2(a\otimes b)\\[3pt]
&= \sum \varepsilon (S(a_{(4)})) \varepsilon (b_{(2)})P_2(1_H\otimes S(a_{(3)})b_{(3)})\otimes S(a_{(1)}S(a_{(5)})b_{(1)}).(a_{(2)}\otimes b_{(4)})\\[3pt]
&=\sum P_2( 1_H\otimes S(a_{(3)})b_{(2)})\otimes S(a_{(1)}S(a_{(4)})b_{(1)}).(a_{(2)}\otimes b_{(3)})\\[3pt]
&=\sum S(S(a_{(4)})b_{(2)}).(1_H\otimes S(a_{(3)})b_{(3)})\otimes S(a_{(1)}S(a_{(5)})b_{(1)}).(a_{(2)}\otimes b_{(4)})\\[3pt]
&=\Delta_sP_2(a\otimes b),
\end{align*}and
\begin{align*}
&(\mathrm{id}\otimes P_2)\Delta_sP_2(a\otimes b) \\[3pt]
&=\sum S(S(a_{(4)})b_{(2)}).(1_H\otimes S(a_{(3)})b_{(3)})\otimes \varepsilon(a_{(1)}S(a_{(5)})b_{(1)})P_2(a_{(2)}\otimes b_{(4)})\\[3pt]
&=\sum S(S(a_{(3)})b_{(1)}).(1_H\otimes S(a_{(2)})b_{(2)})\otimes P_2(a_{(1)}\otimes b_{(3)})\\[3pt]
&=(P_2\otimes P_2)\Delta_s(a\otimes b).
\end{align*}
Therefore
$$(P_2\otimes P_2)\Delta_s(a\otimes b)=
(\mathrm{id}\otimes
P_2+P_2\otimes\mathrm{id}-\mathrm{id}\otimes\mathrm{id})\Delta_sP_2(a\otimes
b) $$ holds for any $a,b\in H$.

Finally,
\begin{align*}
P_2^2(a\otimes b)&=P_2(\sum
S(a_{(1)}S(a_{(3)})b_{(1)}).(a_{(2)}\otimes
b_{(2)}))\\[3pt]
&=\sum \varepsilon(S(a_{(1)}S(a_{(3)})b_{(1)}))P_2(a_{(2)}\otimes
b_{(2)})\\[3pt]
&=\sum
\varepsilon(a_{(1)})\varepsilon(a_{(3)})\varepsilon(b_{(1)})P_2(a_{(2)}\otimes
b_{(2)})\\[3pt]
&=P_2(a\otimes b).
\end{align*}\end{proof}

\begin{remark}We mention here that the operators $P_1$ and $P_2$ above are not of the
type in Proposition \ref{Idempotent RB coalgebras} since they do not satisfy $\Delta  P=(P\otimes
P)  \Delta $.\end{remark}

All constructions given above are idempotent. Now we provide one which
is not in such a case.

\begin{example}\label{Non-idempotent RB coalgebra}
Let $C$ be a $\mathbb{K}$-vector space with a basis
$\{c_n\}_{n\geq 0}$. Define a $\mathbb{K}$-linear map $\Delta:
C\rightarrow C\otimes C$ by
$$\Delta(c_n)=\sum_{j=0}^{n}\sum_{i=n-j}^{n}(-1)^{i+j+n}
\binom{n}{j}\binom{j}{n-i}c_i\otimes c_j,$$ and a
$\mathbb{K}$-linear map $\varepsilon:C\rightarrow \mathbb{K}$ by
$$\varepsilon(c_n)=\begin{cases}
  1,&\text{if }n=0,\\[3pt]
  0,&\text{if }n>0.
\end{cases}$$
Let $c_{-1}=0$. We define a $\mathbb{K}$-linear map
$P:C\rightarrow C$ by
 $P(c_n)=c_{n-1}$.
Then $(C,\Delta,\varepsilon,P)$ is a Rota-Baxter coalgebra of
weight $-1$. The verification is given in the
appendix.\end{example}

\section{Basic properties}

We first discuss the relations between Rota-Baxter algebras and
Rota-Baxter coalgebras. For a given vector space $V$, we denote by
$V^\ast$ the linear dual of $V$, i.e.,
$V^\ast=\mathrm{Hom}_\mathbb{K}(V,\mathbb{K})$. An endomorphism
$\varphi$ of $V$ induces an endomorphism $\varphi^\ast$ of
$V^\ast$ given by $\varphi^\ast(f)=f\varphi$ for any $f\in
V^\ast$. The canonical pairing between $V^\ast$ and $V$ is denoted by $\langle, \rangle$. It induces a pairing, still denoted by $\langle, \rangle$, between $(V^\ast)^{\otimes k}$and $V^{\otimes k}$ for each positive integer $k$ in a natural way. It is well-known that the linear dual of a coalgebra is
an algebra. So we have

\begin{proposition}The linear dual of a (noncounitary) Rota-Baxter coalgebra together with the linear dual of the Rota-Baxter operator is a (nonunitary)
Rota-Baxter algebra. \end{proposition}
\begin{proof}It follows immediately from that\begin{align*}
\Delta^\ast(P^\ast\otimes P^\ast) &=((P\otimes P) \Delta)^\ast\\[3pt]
&=((\mathrm{id}\otimes P)  \Delta   P+(P\otimes \mathrm{id})
\Delta
P+\lambda \Delta  P)^\ast\\[3pt]
&=P^\ast\Delta ^\ast(\mathrm{id}\otimes P^\ast)+P^\ast\Delta
^\ast(P^\ast\otimes \mathrm{id})+\lambda P^\ast\Delta ^\ast.
\end{align*}\end{proof}

But the converse is not true in general, unless the dimension is
finite. In order to overcome this difficulty we need some tools
from the theory of Rota-Baxter algebras and that of coalgebras.

Let $(R,P)$ be a Rota-Baxter algebra of weight $\lambda$. The
\emph{double product} is defined by $$x\star_P
y=xP(y)+P(x)y+\lambda xy,\quad\text{for}~x,y\in R.$$ By the
Theorem 1.1.17 in \cite{Guo}, $R$ equipped with the product
$\star_P$ is a nonunitary $\mathbb{K}$-algebra which will be
denoted by $R_P$, and $(R,\star_P,P)$ is a Rota-Baxter algebra of
weight $\lambda$. Moreover $P$ is a homomorphism of algebras from
$R_P$ to $R$.

Recall that, for any unitary algebra $(A,m,u)$, the finite dual of
 $A$ is $${A}^\circ=\{f\in {A}^*\mid \mathrm{ker}f~\text{contains a cofinite ideal of}~ A\}.$$
A cofinite ideal $I$ of $A$ is an ideal of $A$ subject to
$\dim(A/I)<\infty$. It is well-known that $A^\circ$ is a coalgebra
with comultiplication $\Delta=m^*$ and counit $\varepsilon=u^*$
(see \cite{Sw}).

Under the notation above, we have the following result:

\begin{theorem}Let $(R,P)$ be a unitary Rota-Baxter algebra of weight $\lambda$. If any ideal of $R_P$ is also an ideal of $R$,
then the finite dual $({R}^\circ,P^*)$ is a Rota-Baxter coalgebra
of weight $\lambda$.
\end{theorem}
\begin{proof}
For any cofinite ideal $I$ of $R_P$, $I$ is also a cofinite ideal
of $R$ by the assumption, and hence ${R_P}^\circ\subset R^\circ$.
Since $P$ is an algebra map from $R_P$ to $R$, we get
$P^*(R^\circ)\subset {R_P}^\circ$ by Lemma 6.0.1 (a) in \cite{Sw}.
Therefore $P^*(R^\circ)\subset {R}^\circ$. For any $f\in R^\circ$
and $a,b\in R$, we have
\begin{align*}
\langle(P^*\otimes P^*)\Delta(f),a\otimes b\rangle
&=\langle\Delta(f),P(a)\otimes P(b)\rangle\\[3pt]
& =\langle
f,P(a)P(b)\rangle\\[3pt]
&=\langle f,P(P(a)b+aP(b)+\lambda ab)\rangle\\[3pt]
&=\langle P^*(f),P(a)b+aP(b)+\lambda ab\rangle  \\[3pt]
&=\langle\Delta P^*(f), P(a)\otimes b+a\otimes P(b)+\lambda a\otimes b\rangle\\[3pt]
&=\langle ((P^*\otimes \mathrm{id})\Delta P^*+(\mathrm{id}\otimes
P^*)\Delta P^*+\lambda \Delta P^*)(f),a\otimes b\rangle,
\end{align*}
where the third equality holds since $P$ is a Rota-Baxter operator. So the linear operator $P^*:R^\circ\rightarrow R^\circ$
satisfies $$(P^*\otimes P^*)\Delta=(P^*\otimes \mathrm{id})\Delta
P^*+(\mathrm{id}\otimes P^*)\Delta P^*+\lambda \Delta P^*.$$ Hence
the coalgebra $(R^\circ,P^*)$ is a Rota-Baxter coalgebra of weight
$\lambda$.
\end{proof}

\begin{example}Let $R=t\mathbb{K}[t]$ be the algebra of polynomials without constant terms and $P$ the linear endomorphism of $R$ given by $P(t^n)=\frac{q^n}{1-q^n}t^n,$ where $q\in \mathbb{K}$ is not a root of unity. Then $(R,P)$ is a Rota-Baxter algebra of weight $1$ (see Example 1.1.8 in \cite{Guo}). From \begin{align*}
  t^n\star_P t^m&=P(t^n)t^m+t^nP(t^m)+t^{m+n}\\
&=\frac{q^n}{1-q^n}t^{m+n}+\frac{q^m}{1-q^m}t^{m+n}+t^{m+n}\\
&=\frac{1-q^{m+n}}{(1-q^n)(1-q^m)}t^{m+n},
\end{align*}
we get
\begin{align*}
  t^n t^m=t^n\star_P \frac{(1-q^n)(1-q^m)}{1-q^{m+n}}t^m.
\end{align*}
So for any $a\in R$, there exists $a'\in
R$ such that $t^n a=t^n\star_P a'$. Furthermore,
we can obtain that,  for any $x,a\in R$, there exists
$a'\in R$ such that $x a=x\star_P a'$.
Let $I$ be an ideal of $R_P$. For any $x\in I$ and $a\in R$, there
exists $a'\in R$ such that $xa=x\star_P a'$. Since $x\star_P a'\in
I$, then $xa\in I$. Hence $I$ is also an ideal of
$R$. So the finite dual $(R^\circ, P^*)$ is a Rota-Baxter coalgebra of weight $1$.\end{example}

Let $R$ be a locally finite graded algebra, i.e., $R$ has a direct
decomposition $R=\bigoplus_{i\geq 0}R_i$ of finite dimensional
subspace $R_i$ and $R_iR_j\subset R_{i+j}$. Then the graded dual
$R^g=\bigoplus_{i\geq 0}R_i^\ast$ is a coalgebra (see e.g.,
\cite{Sw}). If moreover $R$ is a Rota-Baxter algebra of weight
$\lambda$ with Rota-Baxter operator $P$ which preserves the grading
$P(R_i)\subset R_i$, then we have

\begin{proposition}The graded dual $R^g$ of $R$ is a Rota-Baxter coalgebra of weight $\lambda$ together with the operator $P^\ast$. \end{proposition}

Now we turn to consider the relation between Rota-Baxter
coalgebras and noncounitary ones. Obviously every Rota-Baxter coalgebra is
noncounitary if we forget the counit. The reverse construction is not so
easy. We provide a solution for the idempotent case in the
following proposition.

\begin{proposition}Let $(C,\Delta, P)$ be an idempotent noncounitary Rota-Baxter coalgebra of
weight $-1$. Then there exists an idempotent Rota-Baxter coalgebra
$(\widetilde{C}, \widetilde{\Delta},\varepsilon,\widetilde{P})$
such that $C=\ker \varepsilon$ as a subspace and the restriction
of $\widetilde{P}$ on $C$ is just $P$.
\end{proposition}
\begin{proof} We extend $(C,\Delta)$ to a coalgebra
$(\widetilde{C}, \widetilde{\Delta},\varepsilon)$ as follows. Fix
a symbol $\mathbf{1}$ and set
$\widetilde{C}=\mathbb{K}\mathbf{1}\oplus C$ as a vector space. We
define $\widetilde{\Delta}(\mathbf{1})=\mathbf{1}\otimes
\mathbf{1}$ and $\widetilde{\Delta}(x)=\Delta(x)+\textbf{1}\otimes
x+x\otimes \mathbf{1}$ for any $x\in C$. The counit $\varepsilon$
is defined to be the projection from $\widetilde{C}$ to
$\mathbb{K}\mathbf{1}$. It is not hard to verify that
$(\widetilde{C}, \widetilde{\Delta},\varepsilon)$ is a coalgebra.
We extend $P$ to $\widetilde{P}$ by setting that
$\widetilde{P}(\mathbf{1})=\mathbf{1}$ and $\widetilde{P}(x)=P(x)$
for $x\in C$. Then for any $x\in C$, we have\begin{align*}
\lefteqn{\big((\mathrm{id}\otimes \widetilde{P})
\widetilde{\Delta} \widetilde{P}+(\widetilde{P}\otimes
\mathrm{id}) \widetilde{\Delta} \widetilde{P}- \widetilde{\Delta}
\widetilde{P}
\big)(\mathbf{1}+x)}\\[3pt]
&=\big((\mathrm{id}\otimes \widetilde{P})  \widetilde{\Delta}
\widetilde{P}+(\widetilde{P}\otimes \mathrm{id})
\widetilde{\Delta}   \widetilde{P}- \widetilde{\Delta}
\widetilde{P} \big)(\mathbf{1})
\\[3pt]
&\ \ \ +\big((\mathrm{id}\otimes \widetilde{P})  \widetilde{\Delta}
\widetilde{P}+(\widetilde{P}\otimes \mathrm{id})
\widetilde{\Delta}   \widetilde{P}- \widetilde{\Delta}
\widetilde{P}
\big)(x)\\[3pt]
&=\mathbf{1}\otimes \mathbf{1}+(\mathrm{id}\otimes \widetilde{P})  \Delta   \widetilde{P}(x)+(\mathrm{id}\otimes \widetilde{P})(\mathbf{1}\otimes \widetilde{P}(x)+\widetilde{P}(x)\otimes \mathbf{1})\\[3pt]
&\ \ \ +(\widetilde{P}\otimes \mathrm{id})  \Delta   \widetilde{P}(x)+(\widetilde{P}\otimes \mathrm{id})(\mathbf{1}\otimes \widetilde{P}(x)+\widetilde{P}(x)\otimes \mathbf{1})\\[3pt]
&\ \ \ -\Delta \widetilde{ P}(x)-\mathbf{1}\otimes \widetilde{P}(x)-\widetilde{P}(x)\otimes \mathbf{1}\\[3pt]
&=\mathbf{1}\otimes \mathbf{1}+(\mathrm{id}\otimes P)  \Delta   P(x)+\mathbf{1}\otimes P^2(x)+P(x)\otimes \widetilde{P}(\mathbf{1})\\[3pt]
&\ \ \ +(P\otimes \mathrm{id})  \Delta   P(x)+\widetilde{P}(\mathbf{1})\otimes P(x)+P^2(x)\otimes \mathbf{1}\\[3pt]
&\ \ \ -\Delta  P(x)-\mathbf{1}\otimes P(x)-P(x)\otimes \mathbf{1}\\[3pt]
&=\mathbf{1}\otimes \mathbf{1}+(\mathrm{id}\otimes P)  \Delta   P(x)+\mathbf{1}\otimes P(x)+P(x)\otimes \mathbf{1}\\[3pt]
&\ \ \ +(P\otimes \mathrm{id})  \Delta   P(x)+\mathbf{1}\otimes P(x)+P(x)\otimes \mathbf{1}\\[3pt]
&\ \ \ -\Delta  P(x)-\mathbf{1}\otimes P(x)-P(x)\otimes \mathbf{1}\\[3pt]
&=\mathbf{1}\otimes \mathbf{1}+\big((\mathrm{id}\otimes P)  \Delta   P+(P\otimes \mathrm{id})  \Delta   P-\Delta  P\big)(x)\\[3pt]
&\ \ \ +\mathbf{1}\otimes P(x)+P(x)\otimes \mathbf{1}\\[3pt]
&=\mathbf{1}\otimes \mathbf{1}+\big(({P}\otimes {P}) {\Delta}\big)(x)+\mathbf{1}\otimes P(x)+P(x)\otimes \mathbf{1}\\[3pt]
&=\big((\widetilde{P}\otimes \widetilde{P}) \widetilde{\Delta}\big)(\mathbf{1})+\big((\widetilde{P}\otimes \widetilde{P}) \widetilde{\Delta}\big)(x)\\[3pt]
&=\big((\widetilde{P}\otimes \widetilde{P})
\widetilde{\Delta}\big)(\mathbf{1}+x).
\end{align*}So $(\widetilde{C}, \widetilde{\Delta},\varepsilon,
\widetilde{P})$ is an idempotent Rota-Baxter coalgebra of weight
$-1$.\end{proof}

Suppose $(C,\Delta,\varepsilon)$ is a coalgebra. A subspace $J$ of
$C$ is called a \emph{noncounitary coideal} if $\Delta(J)\subset C\otimes
J+J\otimes C$. If moreover $\varepsilon(J)=0$, then $J$ is a coideal of
$C$.

\begin{proposition}Let $(C,\Delta,P)$ be a Rota-Baxter coalgebra of nonzero weight $\lambda$. Then $P(C)$ is a noncounitary coideal of $C$. If $P(C)\subset\ker\varepsilon$ then $P(C)$ is a coideal of $C$. Furthermore, the quotient (noncounitary) coalgebra $C/P(C)$ inherits a (noncounitary) Rota-Baxter coalgebra structure. \end{proposition}
\begin{proof}By the equation \eqref{RBO1}, we have $$\Delta  P=\frac{1}{\lambda}\big((P\otimes P) \Delta-(\mathrm{id}\otimes P)  \Delta   P-(P\otimes \mathrm{id})  \Delta   P\big), $$ which implies that  $\Delta(P(C))\subset C\otimes
P(C)+P(C)\otimes C$. The rest of the statements is
straightforward.\end{proof}

\begin{proposition}Let $(C,\Delta,\varepsilon,P)$ be a Rota-Baxter coalgebra of weight $\lambda$. Define $\overline{P}=-\lambda\mathrm{id}-P$. Then $(C,\Delta,\varepsilon,\overline{P})$ is again a Rota-Baxter coalgebra of weight $\lambda$. \end{proposition}
\begin{proof}By the definition, we have\begin{align*}
\lefteqn{(\mathrm{id}\otimes \overline{P})  \Delta
\overline{P}+(\overline{P}\otimes
\mathrm{id})  \Delta   \overline{P}+\lambda \Delta  \overline{P}}\\[3pt]
&=(\mathrm{id}\otimes (-\lambda\mathrm{id}-P))  \Delta
(-\lambda\mathrm{id}-P)+((-\lambda\mathrm{id}-P)\otimes
\mathrm{id})  \Delta   (-\lambda\mathrm{id}-P)\\[3pt]
&\ \ \ +\lambda \Delta  (-\lambda\mathrm{id}-P)\\[3pt]
&=\lambda^2\Delta+\lambda\Delta P+\lambda(\mathrm{id}\otimes P)
\Delta +(\mathrm{id}\otimes P) \Delta
P\\[3pt]
&\ \ \ +\lambda^2\Delta+\lambda\Delta P+\lambda(P\otimes \mathrm{id})
\Delta +(P\otimes \mathrm{id}) \Delta
P-\lambda^2\Delta-\lambda\Delta P\\[3pt]
&=\lambda^2\Delta+\lambda(\mathrm{id}\otimes P)
\Delta+\lambda(P\otimes \mathrm{id}) \Delta +(\mathrm{id}\otimes
P) \Delta P +(P\otimes \mathrm{id}) \Delta
P+\lambda\Delta P\\[3pt]
&=\lambda^2\Delta+\lambda(\mathrm{id}\otimes P)
\Delta+\lambda(P\otimes \mathrm{id}) \Delta +(P\otimes P)\Delta\\[3pt]
&=\big( (-\lambda\mathrm{id}-P)\otimes  (-\lambda\mathrm{id}-P)\big)\Delta\\[3pt]
&=(\overline{P}\otimes \overline{P}) \Delta.
\end{align*}\end{proof}

Corresponding to the double product of Rota-Baxter algebra, we
have the double coproduct construction.

\begin{proposition}\label{prop delta_P}Let $(C,\Delta, P)$ be a noncounitary Rota-Baxter coalgebra of weight $\lambda$. We define $\Delta_P=(P\otimes \mathrm{id})  \Delta+(\mathrm{id}\otimes P)  \Delta+\lambda \Delta$. Then $\Delta_PP=(P\otimes P)\Delta$ and $(C,\Delta_P, P)$ is again a noncounitary Rota-Baxter coalgebra of weight $\lambda$. \end{proposition}
\begin{proof}
The equality
\begin{equation*}
 \Delta_PP=(P\otimes P)\Delta
\end{equation*} follows from \eqref{RBO1} immediately. We use it to show the coassociativity of $\Delta_P$:
\begin{align*}
&(\Delta_P\otimes \mathrm{id}) \Delta_P -(\mathrm{id}\otimes \Delta_P) \Delta_P\\[3pt]
&=(\Delta_P\otimes \mathrm{id})(P\otimes \mathrm{id}+\mathrm{id}\otimes P+\lambda\mathrm{id}\otimes \mathrm{id})\Delta \\[3pt]
&\quad -(\mathrm{id}\otimes \Delta_P)(P\otimes \mathrm{id}+\mathrm{id}\otimes P+\lambda\mathrm{id}\otimes \mathrm{id})\Delta\\[3pt]
&=(\Delta_P P\otimes \mathrm{id}+\Delta_P\otimes P+\lambda \Delta_P\otimes \mathrm{id})\Delta\\[3pt]
&\quad-(P\otimes \Delta_P+\mathrm{id}\otimes \Delta_P P +\lambda \mathrm{id}\otimes \Delta_P )\Delta\\[3pt]
&=\big(((P\otimes P)\Delta)\otimes \mathrm{id} +((P\otimes \mathrm{id}+\mathrm{id}\otimes P+\lambda\mathrm{id}\otimes \mathrm{id})\Delta)\otimes P \\[3pt]
&\quad+((\lambda P\otimes \mathrm{id}+\lambda\mathrm{id}\otimes P+\lambda^2\mathrm{id}\otimes \mathrm{id})\Delta)\otimes \mathrm{id} \big)\Delta\\[3pt]
&\quad-\big(P\otimes ((P\otimes \mathrm{id}+\mathrm{id}\otimes P+\lambda\mathrm{id}\otimes \mathrm{id})\Delta)+\mathrm{id}\otimes((P\otimes P)\Delta)\\[3pt]
&\quad+\lambda\mathrm{id}\otimes ((P\otimes \mathrm{id}+\mathrm{id}\otimes P+\lambda\mathrm{id}\otimes \mathrm{id})\Delta)\big)\Delta\\[3pt]
&=(P\otimes P\otimes\mathrm{id}+P\otimes \mathrm{id}\otimes P+\mathrm{id}\otimes P\otimes P+\lambda \mathrm{id}\otimes\mathrm{id}\otimes P\\[3pt]
&\quad +\lambda P\otimes \mathrm{id}\otimes \mathrm{id}+\lambda \mathrm{id}\otimes P\otimes \mathrm{id}+\lambda^2 \mathrm{id}\otimes \mathrm{id}\otimes \mathrm{id})(\Delta\otimes \mathrm{id})\Delta\\[3pt]
&\quad-(P\otimes P\otimes \mathrm{id}+P\otimes \mathrm{id}\otimes P+\lambda P\otimes\mathrm{id}\otimes \mathrm{id}+\mathrm{id}\otimes P\otimes P\\[3pt]
&\quad+\lambda \mathrm{id}\otimes P\otimes \mathrm{id}+\lambda \mathrm{id}\otimes \mathrm{id}\otimes P+\lambda^2\mathrm{id}\otimes \mathrm{id}\otimes\mathrm{id})(\mathrm{id}\otimes\Delta)\Delta\\[3pt]
&=0.
\end{align*}

Finally, we have
\begin{align*}
\lefteqn{(\mathrm{id}\otimes P)  \Delta_P   P+(P\otimes
\mathrm{id})  \Delta_P   P+\lambda \Delta_P  P}\\[3pt]
&=(\mathrm{id}\otimes P) (P\otimes P)\Delta +(P\otimes
\mathrm{id})  (P\otimes P)\Delta +\lambda (P\otimes P)\Delta \\[3pt]
&=(P\otimes P)((P\otimes \mathrm{id})  \Delta+(\mathrm{id}\otimes P)  \Delta+\lambda \Delta)\\[3pt]
&=(P\otimes P) \Delta_P.
\end{align*}\end{proof}

\begin{proposition}Let $(C,\Delta)$ be a (noncounitary) coalgebra. A linear operator $P$ on $C$ is an idempotent Rota-Baxter operator of weight $-1$ if and only if there is a $\mathbb{K}$-vector space direct sum decomposition $C=C_1\oplus C_2$ of $C$ into noncounitary coideals $C_1$ and $C_2$ such that $$P:C\rightarrow C_1$$ is the projection of $C$ onto $C_1$: $P(c_1+c_2)=c_1$ for $c_1\in C_1$ and $c_2\in C_2$.
\end{proposition}
\begin{proof}
Suppose $C$ has a $\mathbb{K}$-vector space direct sum
decomposition $C=C_1\oplus C_2$, where $C_1$ and $C_2$ are
noncounitary coideals of $C$. Then the projection $P$ of $C$ onto
$C_1$ is idempotent since, for $c=c_1+c_2$ in $C$ with $c_1\in
C_1$  and $c_2\in C_2$, we have
$$P^2(c)=P^2(c_1+c_2)=P(c_1)=c_1=P(c).$$
Further, for $c=c_1+c_2$ in $C$ with $c_1\in C_1$  and $c_2\in C_2$, we have
$\Delta(c_1)=\sum_i a_i'\otimes x_i+\sum_j y_j\otimes a_j''$ and $\Delta(c_2)=\sum_i b_i'\otimes u_i+\sum_j v_j\otimes b_j''$, where $a_i',a_j''\in C_1$, $b_i',b_j''\in C_2$ and $x_i,y_j,u_i,v_j\in C$. Then
\begin{align*}
&(\mathrm{id}\otimes P+P\otimes \mathrm{id}-\mathrm{id}\otimes\mathrm{id})\Delta(P(c))\\[3pt]
&=(\mathrm{id}\otimes P+P\otimes \mathrm{id}-\mathrm{id}\otimes\mathrm{id})\Delta(c_1)\\[3pt]
&=\sum_i a_i'\otimes P(x_i)+\sum_j y_j\otimes a_j''+\sum_i a_i'\otimes x_i+\sum_j P(y_j)\otimes a_j''\\[3pt]
&~~-\sum_i a_i'\otimes x_i-\sum_j y_j\otimes a_j''\\[3pt]
&=\sum_i a_i'\otimes P(x_i)+\sum_j P(y_j)\otimes a_j''\\[3pt]
&=(P\otimes P)\Delta(c).\end{align*} Hence $P$ is an idempotent
Rota-Baxter operator of weight $-1$.

Conversely, suppose $P:C\rightarrow C $ is an idempotent Rota-Baxter operator of weight $-1$. Let $C_1=P(C)$ and $C_2=(\mathrm{id}-P)(C)$. It is easy to see that $C=C_1\oplus C_2$ as vector spaces. Since $c=P(c)+(\mathrm{id}-P)(c)$ is the decomposition of $c\in C$, we see that $P$ is the projection of $C$ onto $C_1$.
For any $c_1\in C_1$, from
$$(P\otimes P)\Delta(c_1)=(\mathrm{id}\otimes P+P\otimes \mathrm{id}-\mathrm{id}\otimes\mathrm{id})\Delta(P(c_1)),$$
we have
\begin{align*}
  \Delta(c_1)&=(\mathrm{id}\otimes P+P\otimes \mathrm{id}-P\otimes P)\Delta(c_1)\in C\otimes C_1+C_1\otimes C.
\end{align*}
So $\Delta(C_1)\subset  C\otimes C_1+C_1\otimes C$.

 For any $c_2\in C_2$, we have $$(P\otimes P)\Delta(c_2)=(\mathrm{id}\otimes P+P\otimes \mathrm{id}-\mathrm{id}\otimes\mathrm{id})\Delta(P(c_2))=0,$$
so $\Delta(C_2)\subset \mathrm{ker} P\otimes C+C\otimes
\mathrm{ker} P=C_2\otimes C+C\otimes C_2$. Hence $C_1$ and $C_2$
are noncounitary coideals of $C$.\end{proof}

We now give a dual version of the additive decomposition of Atkinson \cite{At}.

\begin{proposition}
  Let $(C,\Delta, \varepsilon)$ be a coalgebra. If a linear operator $P:C\rightarrow C$ is a Rota-Baxter operator of weight $\lambda$, then there is a linear map $\Phi : C\rightarrow C\otimes C$ such that $$(P\otimes P)\Delta=\Phi  P\qquad \text{and} \qquad (\overline{P}\otimes \overline{P})\Delta=-\Phi \overline{P},$$
  where $\overline{P}=-\lambda\mathrm{id}-P$. Suppose $\lambda\in \mathbb{K}$ is not zero. Then the converse is also true.
\end{proposition}
\begin{proof}

 If $P:C\rightarrow C$ is a Rota-Baxter operator of weight $\lambda$, then setting
 $$\Phi=(\mathrm{id}\otimes P+P\otimes\mathrm{id}+\lambda\mathrm{id}\otimes \mathrm{id})\Delta,$$ we get $(P\otimes P)\Delta=\Phi  P$ and
 \begin{align*}
(\overline{P}\otimes \overline{P})\Delta&=\big( (-\lambda\mathrm{id}-P)\otimes  (-\lambda\mathrm{id}-P)\big)\Delta\\[3pt]
  &= \lambda^2\Delta+\lambda(\mathrm{id}\otimes P)
\Delta+\lambda(P\otimes \mathrm{id}) \Delta +(P\otimes P)\Delta\\[3pt]
&=\lambda\Phi+ (P\otimes P)\Delta=\lambda\Phi+ \Phi P\\&=\Phi(\lambda\mathrm{id}+P)=-\Phi \overline{P}.
 \end{align*}

 Conversely, we have
 \begin{align*}
   \lambda \Phi&=-\Phi\overline{P}-\Phi P\\[3pt]
   &=(\overline{P}\otimes \overline{P})\Delta- (P\otimes P)\Delta\\[3pt]
   &=\big( (-\lambda\mathrm{id}-P)\otimes  (-\lambda\mathrm{id}-P)\big)\Delta- (P\otimes P)\Delta\\[3pt]
   &=\lambda\big(\lambda\Delta+(\mathrm{id}\otimes P)
\Delta+(P\otimes \mathrm{id}) \Delta \big).
 \end{align*}
 Since $\lambda\neq0$, we get $$\Phi=\lambda\Delta+(\mathrm{id}\otimes P)
\Delta+(P\otimes \mathrm{id}) \Delta.$$ So
 $(P\otimes P)\Delta=\Phi P$ implies that $P$ is a Rota-Baxter operator of weight $\lambda$.
 \end{proof}

\section{Rota-Baxter comodules}

\begin{definition}
  Let $(C,\Delta,P)$ be a Rota-Baxter coalgebra of weight $\lambda$.
\begin{enumerate}
  \item
  A \emph{left Rota-Baxter comodule} over $(C,\Delta,P)$ or simply a \emph{left $(C,\Delta,P)$-comodule} is a triple $(M,p,\delta)$ where $(M,\delta)$ is a $C$-comodule and a linear map $p:M\rightarrow M$ satisfying
  $$(P\otimes p)\delta=(\mathrm{id}\otimes p)\delta p+ (P\otimes \mathrm{id})\delta p+\lambda \delta p.$$
  In Sweedler's notation, this means
  \begin{align*}
   \sum P(m_{(-1)})&\otimes p(m_{(0)})=
    \sum p(m)_{(-1)}\otimes p(p(m)_{(0)})\\& + \sum P(p(m)_{(-1)})\otimes p(m)_{(0)} +\lambda \sum p(m)_{(-1)}\otimes p(m)_{(0)}
  \end{align*}
  for any $m\in M$.

  \item Let $(M,p_M,\delta_M)$ and $(N,p_N,\delta_N)$ be two $(C,\Delta,P)$-comodules. A \emph{homomorphism $f:(M,p_M,\delta_M)\rightarrow (N,p_N,\delta_N)$ of Rota-Baxter modules} is a homomorphism $f:M\rightarrow N$ of $C$-comodules such that $f\circ p_M=p_N\circ f$. Denote $\mathrm{Hom}_{(C,\Delta, P)}(M,N)$ for the set of all $(C,\Delta,P)$-comodule homomorphisms, which is a subspace of $\mathrm{Hom}_C(M,N)$ the space of all $C$-comodule homomorphisms.

\end{enumerate}\end{definition}

\begin{remark}(1) In \cite{GL}, the notion of Rota-Baxter module over a Rota-Baxter algebra is introduced. More precisely, a left Rota-Baxter module $(M, p)$ over a Rota-Baxter algebra $(R,P)$ of weight $\lambda$ is an $R$-module $M$ together with a linear endomorphism $p$ of $M$ such that $$P(a)p(x)=p(ap(x))+p(P(a)x)+\lambda p(ax),$$for any $a\in R$ and $x\in M$. One can see immediately that our notion of Rota-Baxter comodule is the dual of Rota-Baxter modules.

(2) One can define similarly the notion of right Rota-Baxter comodule over a Rota-Baxter coalgebra as follows. It is a right $C$-comodule $(M, \rho)$ together with an endomorphism $p$ of $M$ such that
  \begin{equation}\label{Right comodules}(p\otimes P)\rho=(p\otimes \mathrm{id})\rho p+ (\mathrm{id}\otimes P)\rho p+\lambda \rho p.\end{equation}

  (3) We have shown that for any Rota-Baxter coalgebra $(C,\Delta, P)$ of weight $\lambda$, the triple $(C,\Delta, \overline{P})$ with $\overline{P}=-\lambda \mathrm{id}_C-P$ is also a Rota-Baxter coalgebra of weight $\lambda$. Similarly, if $(M,p,\delta)$ is a $(C,\Delta,P)$-comodule, then
$(M,\overline{p},\delta)$ is a $(C,\Delta,\overline{P})$-comodule with $\overline{p}=-\lambda\mathrm{id}_M-p$. Furthermore, if $f:(M, p_M,\delta_M)\rightarrow (N,p_N,\delta_N)$ is a $(C,\Delta,P)$-comodule, then $f$ is also a $(C,\Delta,\overline{P})$-comodule homomorphism from $(M,\overline{p}_M,\delta_M)$ to $(N,\overline{p}_N,\delta_N)$.\end{remark}

\begin{example}\label{Example of RB comodules}A Rota-Baxter coalgebra $(C,\Delta,P)$ is naturally a left or right Rota-Baxter comodule over $(C,\Delta,P)$ whose comodule structure is given by the coproduct.
And any $C$-comodule $(M,\delta)$ is automatically a Rota-Baxter comodule over $(C,\Delta,P)$ with $p_M=0$.
\end{example}

Now let $(M,p,\delta)$ be a $(C,\Delta,P)$-comodule. We define a new linear map  $\delta_p: M\rightarrow C\otimes M$ by
\begin{align*}
\delta_p=(P\otimes \mathrm{id}+\mathrm{id}\otimes p+\lambda \mathrm{id}\otimes \mathrm{id})\delta.
\end{align*}

Using the same argument as in the proof of Proposition \ref{prop delta_P}, we can get the following proposition.

\begin{proposition}
  Let $(C,\Delta,P)$ be a Rota-Baxter coalgebra and $(M,p,\delta)$ a $(C,\Delta,P)$-comodule.
  Then $$\delta_p p=(P\otimes p)\delta $$ and
  $(M,\delta_p, p)$ is a Rota-Baxter $(C,\Delta_P,P)$-comodule.

%
%
%

\end{proposition}

\begin{definition}Let $(M,p,\delta)$ be a left Rota-Baxter comodule over a Rota-Baxter coalgebra $(C,\Delta,P)$. A linear subspace $N$ of $M$ is called a \emph{Rota-Baxter subcomodule} if $\delta(N)\subset C\otimes N$ and $p(N)\subset N$. \end{definition}

It is easy to see that the quotient of a left Rota-Baxter comodule by a Rota-Baxter subcomodule is again a left Rota-Baxter comodule.

Now let us recall the notion of rational module. Given a coalgebra $C$ and a right $C$-comodule $(M,\rho)$, one has a $C^\ast$-module structure on $M$ as follows. For any $c^\ast\in C^\ast$ and $m\in M$, one defines \begin{equation}\label{rational module}c^\ast.m=\sum c^\ast(m_{(1)})m_{(0)},\end{equation}where $\rho(m)=\sum m_{(0)}\otimes m_{(1)}$. It is called the rational $C^\ast$-module of $(M,\rho)$ (see, e.g., Proposition 2.1.1 in \cite{Sw}).

\begin{proposition}Suppose that $(C,\Delta, P)$ is a Rota-Baxter coalgebra of weight $\lambda$ and $(M,p,\rho)$ is a right $(C,\Delta, P)$-comodule. If $p$ is a comodule morphism, then the rational $C^\ast$-module $M$, equipped with $p$, is a left Rota-Baxter module over the Rota-Baxter algebra $(C^\ast, P^\ast)$.\end{proposition}
\begin{proof}Since $p$ is a comodule morphism, we have $$\sum p(m)_{(0)}\otimes p(m)_{(1)}=\sum p(m_{(0)})\otimes m_{(1)},$$for any $m\in M$. Then for any $c^\ast\in C^\ast$, $$c^\ast.p(m)=\sum c^\ast(m_{(1)})p(m_{(0)}).$$It follows from the equation (\ref{Right comodules}) that   \begin{align*}
\lefteqn{\sum p(m_{(0)})\otimes P(m_{(1)})}\\[3pt]
&=\sum p(p(m)_{(0)})\otimes p(m)_{(1)} + \sum p(m)_{(0)}\otimes P(p(m)_{(1)}) +\lambda \sum p(m)_{(0)}\otimes p(m)_{(1)}\\[3pt]
&=\sum p(p(m_{(0)}))\otimes m_{(1)} + \sum p(m_{(0)})\otimes P(m_{(1)}) +\lambda \sum p(m_{(0)})\otimes m_{(1)}.
  \end{align*}Therefore \begin{align*}
\lefteqn{p(c^\ast .p(m))+p(P^\ast(c^\ast).m)+\lambda p(c^\ast.m)}\\[3pt]
&=p\Big(\sum c^\ast(m_{(1)})p(m_{(0)})\Big)+p\Big(\sum c^\ast(P(m_{(1)}))m_{(0)}\Big)+\lambda p\Big(\sum c^\ast(m_{(1)})m_{(0)}\Big)\\[3pt]
&=\sum c^\ast(m_{(1)})p(p(m_{(0)}))+\sum c^\ast(P(m_{(1)}))p(m_{(0)})+\lambda \sum c^\ast(m_{(1)})p(m_{(0)})\\[3pt]
&=\sum c^\ast(P(m_{(1)}))p(m_{(0)})\\[3pt]
&=\sum P^\ast(c^\ast)(m_{(1)})p(m_{(0)})\\[3pt]
&=P^\ast(c^\ast).p(m).
\end{align*}
\end{proof}

\begin{remark}It is easy to see that the Rota-Baxter comodules in Example \ref{Example of RB comodules} verify the conditions in the above proposition.\end{remark}

\section{Appendix}

In this appendix, we present a detailed proof of Example \ref{Non-idempotent RB coalgebra}. The
techniques used here are similar to \cite{AGKK}.

It is easy to see that $\varepsilon$ is a counit of $C$. Before we
check the coassociativity, we need the following lemma.

\begin{lemma}\label{lemma1} For integers $n,k,l,j$ with $0\leq k,l,j\leq n$, we have
\begin{equation}\label{id4}
\sum_{i=0}^{n}\binom{j}{n-i}\binom{i}{l}\binom{l}{i-k}
=\sum_{i=0}^{n}\binom{n-j}{n-i}\binom{i}{n-k}\binom{j}{i-l}.
\end{equation}
\end{lemma}
\begin{proof}
For brevity, write the left-hand side  and right-hand side of  \eqref{id4} as $L(n,k,l,j)$ and $R(n,k,l,j)$ respectively. We use induction on $n$ with $n\geq 0$.
When $n=0$, then $k=l=j=0$, we can check directly that $$L(0,0,0,0)=R(0,0,0,0)=1.$$ Assume that the equation $L(n-1,k,l,j)=R(n-1,k,l,j)$ holds for any integers $k,l,j,n$ with $0\leq k,l,j\leq n-1$.
Now consider the case for $0\leq k,l,j\leq n$. First we prove some special cases. When $k=n$, we have
\begin{align*}
L(n,n,l,j)=\sum_{i=0}^{n}\binom{j}{n-i}\binom{i}{l}\binom{l}{i-n} =\binom{j}{0}\binom{n}{l}\binom{l}{0}=\binom{n}{l},
\end{align*}
and $$R(n,n,l,j)=\sum_{i=0}^{n}\binom{n-j}{n-i}\binom{j}{i-l}
=\sum_{i=0}^{n}\binom{n-j}{i}\binom{j}{n-l-i}.$$ Using the
classical Vandermonde's identity, for any integers $x,y,z\geq 0$,
$$\sum_{t=0}^{x}\binom{x}{t}\binom{y}{z-t}=\binom{x+y}{z},$$ we
get $$R(n,n,l,j)=\binom{n}{n-l}=L(n,n,l,j).$$ When   $l=n$, we
have
$$L(n,k,n,j)=\sum_{i=0}^{n}\binom{j}{n-i}\binom{i}{n}\binom{n}{i-k} =\binom{j}{0}\binom{n}{n}\binom{n}{n-k}=\binom{n}{k},$$
and
$$R(n,k,n,j)=\sum_{i=0}^{n}\binom{n-j}{n-i}\binom{i}{n-k}\binom{j}{i-n} =\binom{n}{n-k}=L(n,k,n,j).$$
When $j=n$, since
$$\displaystyle\binom{i}{l}\binom{l}{i-k}=\binom{i}{k}\binom{k}{i-l}
~~\text{and}~~
\binom{n}{i}\binom{i}{k}=\binom{n}{k}\binom{n-k}{n-i}, $$ using
the vandermonde's identity again, we have
$$L(n,k,l,n)=\sum_{i=0}^{n}\binom{n}{n-i}\binom{i}{k}\binom{k}{i-l} =\sum_{i=0}^{n}\binom{n}{k}\binom{n-k}{n-i}\binom{k}{i-l} =\binom{n}{k}\binom{n}{l},$$
and
$$R(n,k,l,n)=\sum_{i=0}^{n}\binom{0}{n-i}\binom{i}{n-k}\binom{n}{i-l} =\binom{n}{n-k}\binom{n}{n-l}=L(n,k,l,n).$$
When $j=0$, we have
$$L(n,k,l,0)=\sum_{i=0}^{n}\binom{0}{n-i}\binom{i}{l}\binom{l}{i-k} =\binom{n}{l}\binom{l}{n-k},$$
and $$R(n,k,l,0)=\sum_{i=0}^{n}\binom{n}{n-i}\binom{i}{n-k}\binom{0}{i-l} =\binom{n}{n-l}\binom{l}{n-k}=L(n,k,l,0).$$

Next we will use the Pascal's rule  $$\binom{x}{y}=\binom{x-1}{y-1}+\binom{x-1}{y}$$
where $x,y$ are integers with $x\geq y\geq 0$.

For $0\leq k,l<n$ and $0<j<n$,  we have
\begin{align*}
&L(n,k,l,j)=\sum_{i=0}^n\binom{j}{n-i}\binom{i}{l}\binom{l}{i-k} =\sum_{i=1}^n\binom{j}{n-i}\binom{i}{l}\binom{l}{i-k} \\ &=\sum_{i=1}^{n}\Big[\binom{j-1}{n-i}+\binom{j-1}{n-i-1}\Big] \binom{i}{l}\binom{l}{i-k}\\
&=\sum_{i=1}^{n}\binom{j-1}{n-1-(i-1)}\binom{i-1+1}{l}\binom{l}{i-1-(k-1)}\\ &\quad+\sum_{i=1}^{n}\binom{j-1}{n-i-1}\binom{i}{l}\binom{l}{i-k}\\
&=\sum_{i=0}^{n-1}\binom{j-1}{n-1-i}\binom{i+1}{l}\binom{l}{i-(k-1)} +\sum_{i=0}^{n-1}\binom{j-1}{n-1-i}\binom{i}{l}\binom{l}{i-k}\\
&=\sum_{i=0}^{n-1}\binom{j-1}{n-1-i} \Big[\binom{i}{l-1}+\binom{i}{l}\Big] \binom{l}{i-(k-1)}+L(n-1,k,l,j-1)\\
&=\sum_{i=0}^{n-1}\binom{j-1}{n-1-i} \binom{i}{l-1} \Big[\binom{l-1}{i-k}+\binom{l-1}{i-(k-1)}\Big] \\&\quad +L(n-1,k-1,l,j-1)+L(n-1,k,l,j-1)\\
&=L(n-1,k,l-1,j-1)+L(n-1,k-1,l-1,j-1)\\&\quad +L(n-1,k-1,l,j-1)+L(n-1,k,l,j-1),
  \end{align*}
and \begin{align*}
&R(n,k,l,j)=\sum_{i=0}^{n}\binom{n-j}{n-i}\binom{i}{n-k}\binom{j}{i-l} =\sum_{i=1}^{n}\binom{n-j}{n-i}\binom{i}{n-k}\binom{j}{i-l}\\
&=\sum_{i=1}^{n}\binom{n-1-(j-1)}{n-1-(i-1)}\binom{i-1+1}{n-1-(k-1)} \binom{j-1+1}{i-1-(l-1)}\\
&=\sum_{i=0}^{n-1}\binom{n-1-(j-1)}{n-1-i}\binom{i+1}{n-1-(k-1)} \binom{j-1+1}{i-(l-1)}\\
&=\sum_{i=0}^{n-1}\binom{n-1-(j-1)}{n-1-i} \Big[\binom{i}{n-1-k}+\binom{i}{n-1-(k-1)}\Big] \\ &\quad\times\Big[\binom{j-1}{i-l}+\binom{j-1}{i-(l-1)}\Big]\\
&=R(n-1,k,l,j-1)+R(n-1,k-1,l,j-1)+R(n-1,k,l-1,j-1)\\&\quad+R(n-1,k-1,l-1,j-1).
\end{align*}
By induction, we obtain $L(n,k,l,j)=R(n,k,l,j)$. This proves the lemma.
\end{proof}

Note that for integers $x,y$, we have
\begin{equation}
  \label{fact1}\binom{x}{y}=0\quad \text{if}~x\geq 0>y ~\text{or}~y>x\geq 0.
\end{equation}
Then we can write
$$\Delta(c_n)=\sum_{j=0}^{n}\sum_{i=0}^{n}(-1)^{i+j+n} \binom{n}{j}\binom{j}{n-i}c_i\otimes c_j.$$Now we are ready to check the coassociativity of $(C,\Delta, \varepsilon)$. We have
\begin{align*}
&(\Delta\otimes \mathrm{id})\Delta(c_n) =\sum_{j=0}^{n}\sum_{i=0}^{n}(-1)^{i+j+n} \binom{n}{j}\binom{j}{n-i}\Delta(c_i)\otimes c_j\\
&=\sum_{j=0}^{n}\sum_{i=0}^{n}(-1)^{i+j+n} \binom{n}{j}\binom{j}{n-i} \Big(\sum_{l=0}^{n}\sum_{k=0}^{n}(-1)^{i+k+l} \binom{i}{l}\binom{l}{i-k}c_k\otimes c_l\Big)\otimes c_j\\
&=\sum_{j=0}^{n}\sum_{i=0}^{n}\sum_{l=0}^{n}\sum_{k=0}^{n}(-1)^{j+n+k+l} \binom{n}{j}\binom{j}{n-i}\binom{i}{l}\binom{l}{i-k} c_k\otimes c_l\otimes c_j,
\end{align*}
and
\begin{align*}
&(\mathrm{id}\otimes\Delta)\Delta(c_n) =\sum_{i=0}^{n}\sum_{k=0}^{n}(-1)^{k+i+n} \binom{n}{i}\binom{i}{n-k}c_k\otimes\Delta(c_i)\\
&=\sum_{i=0}^{n}\sum_{k=0}^{n}(-1)^{k+i+n} \binom{n}{i}\binom{i}{n-k} c_k\otimes\Big(\sum_{j=0}^{n}\sum_{l=0}^{n}(-1)^{j+i+l} \binom{i}{j}\binom{j}{i-l}c_l\otimes c_j\Big) \\
&=\sum_{i=0}^{n}\sum_{k=0}^{n}\sum_{j=0}^{n} \sum_{l=0}^{n}(-1)^{j+n+k+l} \binom{n}{i}\binom{i}{n-k}\binom{i}{j} \binom{j}{i-l}c_k\otimes c_l\otimes c_j\\
&=\sum_{i=0}^{n}\sum_{k=0}^{n} \sum_{j=0}^{n}\sum_{l=0}^{n}(-1)^{j+n+k+l} \binom{n}{j}\binom{n-j}{n-i} \binom{i}{n-k}\binom{j}{i-l}c_k\otimes c_l\otimes c_j.
\end{align*}
The second equalities in
the above two equations are consequences of the fact \eqref{fact1}.  The last equality follows from
 $$\binom{n}{i}\binom{i}{j}=\binom{n}{j}\binom{n-j}{n-i}.$$
Due to Lemma \ref{lemma1}, the coassociativity $(\Delta\otimes \mathrm{id})\Delta=(\mathrm{id}\otimes\Delta)\Delta$ holds. Hence the coalgebra $(C,\Delta,\varepsilon)$ is well-defined.

Next we check that $P$ is a Rota-Baxter operator of weight $-1$ on $C$.
Since
$$\Delta   P(c_n) =\Delta(c_{n-1}) =\sum_{j=0}^{n-1}\sum_{i=0}^{n-1}(-1)^{i+j+n-1} \binom{n-1}{j}\binom{j}{n-1-i}c_i\otimes c_j,$$
 then \begin{align*}
&(\mathrm{id}\otimes P+P\otimes \mathrm{id}-\mathrm{id}\otimes \mathrm{id})  \Delta  P(c_n) \\ &=\sum_{j=1}^{n-1}\sum_{i=0}^{n-1}(-1)^{i+j+n-1} \binom{n-1}{j}\binom{j}{n-1-i}c_i\otimes c_{j-1}\\
&\quad+\sum_{j=0}^{n-1}\sum_{i=0}^{n-1}(-1)^{i+j+n-1} \binom{n-1}{j}\binom{j}{n-1-i}c_{i-1}\otimes c_j\\
&\quad-\sum_{j=0}^{n-1}\sum_{i=0}^{n-1}(-1)^{i+j+n-1} \binom{n-1}{j}\binom{j}{n-1-i}c_i\otimes c_j\\
&=\sum_{j=0}^{n-2}\sum_{i=0}^{n-1}(-1)^{i+j+n} \binom{n-1}{j+1}\binom{j+1}{n-1-i}c_i\otimes c_{j}\\
&\quad+\sum_{j=0}^{n-1}\sum_{i=0}^{n-2}(-1)^{i+j+n} \binom{n-1}{j}\binom{j}{n-2-i}c_{i}\otimes c_j\\
&\quad+\sum_{j=0}^{n-1}\sum_{i=0}^{n-1}(-1)^{i+j+n} \binom{n-1}{j}\binom{j}{n-1-i}c_i\otimes c_j\\
&=\sum_{j=0}^{n-1}\sum_{i=0}^{n-1}(-1)^{i+j+n} \Big\{ \binom{n-1}{j+1}\Big[\binom{j}{n-1-i}+\binom{j}{n-2-i)}\Big]\\ &\quad+\binom{n-1}{j}\binom{j}{n-2-i}+\binom{n-1}{j}\binom{j}{n-1-i}\Big\}
c_i\otimes c_j\\
&=\sum_{j=0}^{n-1}\sum_{i=0}^{n-1}(-1)^{i+j+n} \Big[\binom{n-1}{j+1}+\binom{n-1}{j}\Big] \Big[\binom{j}{n-1-i}\binom{j}{n-2-i}\Big]   c_i\otimes c_j\\
&=\sum_{j=0}^{n-1}\sum_{i=0}^{n-1}(-1)^{i+j+n} \binom{n}{j+1}\binom{j+1}{n-1-i} c_i\otimes c_j\\
&=\sum_{j=0}^{n}\sum_{i=0}^{n}(-1)^{i+j+n} \binom{n}{j}\binom{j}{n-i}c_{i-1}\otimes c_{j-1}=(P\otimes P) \Delta(c_n).
\end{align*}
Therefore $(C,\Delta,\varepsilon,P)$ is a Rota-Baxter coalgebra of
weight $-1$.

\section*{Acknowledgements}The authors would like to thank Marc Rosso
for the helpful discussions about coalgebras. This work was
partially supported by National Natural Science Foundation of
China (Grant No. 11201067, No. 11931009, No. 11871249,  No. 11871326).

\bibliographystyle{amsplain}

\begin{thebibliography}{10}
\bibitem{A}M. Aguiar, Pre-Poisson algebras, Lett. Math. Phys., 54 (2000) 263--277.


\bibitem{AB}H. An, C. Bai, From Rota-Baxter algebras to pre-Lie
algebras, J. Phys. A, 41 (2008) 015201, 19 pp.



\bibitem{AGKK}G. E. Andrews, L. Guo, W. Keigher, K. Ono, Baxter algebras and Hopf algebras, Trans. Amer. Math. Soc.,
355 (2003) 4639--4656.

\bibitem{At}F. V. Atkinson, Some aspects of Baxter's functional equation, J. Math. Anal. Appl., 7 (1963) 1--30.


\bibitem{Bax}G. Baxter, An analytic problem whose solution
follows from a simple algebraic identity, Pacific J. Math., 10 (1960) 731--742.

\bibitem{CG}C. Chu, L. Guo, Localization of Rota-Baxter
algebras, J. Pure Appl. Algebra, 218 (2014) 237--251.


\bibitem{CK1}A. Connes, D. Kreimer, Renormalization in quantum field theory and the Riemann-Hilbert problem. I.
The Hopf algebra structure of graphs and the main theorem, Comm. Math. Phys., 210 (2000) 249--273.

\bibitem{CK2}A. Connes, D. Kreimer,
Renormalization in quantum field theory and the Riemann-Hilbert problem. II.
The $\beta$-function, diffeomorphisms and the renormalization group, Comm. Math. Phys., 216
 (2001) 215--241.

\bibitem{E}K. Ebrahimi-Fard, Loday-type algebras and the
Rota-Baxter relation, Lett. Math. Phys., 61 (2002) 139--147.


\bibitem{EGP}K. Ebrahimi-Fard, J. M. Gracia-Bond\'{\i}a, F. Patras,
Rota-Baxter algebras and new combinatorial identities, Lett.
Math. Phys., 81 (2007) 61--75.



\bibitem{EG2}K. Ebrahimi-Fard, L. Guo, Rota-Baxter algebras and
dendriform algebras, J. Pure Appl. Algebra, 212 (2008) 320--339.


\bibitem{EG3}K. Ebrahimi-Fard, L. Guo,
Multiple zeta values and Rota-Baxter algebras, Integers, 8 (2008) 18 pp.


\bibitem{EGK1}K. Ebrahimi-Fard, L. Guo, D. Kreimer, Spitzer's identity
and the algebraic Birkhoff decomposition in pQFT, J. Phys. A, 37 (2004) 11037--11052.


\bibitem{EGK2}K. Ebrahimi-Fard, L. Guo, D. Kreimer, Integrable
renormalization. II. The general case, Ann. Henri. Poincar\'{e},
 6 (2005) 369--395.

\bibitem{G}L. Guo, Properties of free Baxter algebras, Adv. Math., 151 (2000) 346--374.

 \bibitem{Guo}L. Guo, An introduction to Rota-Baxter algebra,
Surveys of Modern Mathematics 4. International Press, Somerville, MA;
 Higher Education Press, Beijing, 2012.


\bibitem{GK}L. Guo, W. Keigher, Baxter algebras and shuffle
products, Adv. Math., 150 (2000) 117--149.

\bibitem{GL}L. Guo, Z. Lin, Representations and modules of Rota-Baxter algebras, arXiv:1905.01531v2.


\bibitem{GZ}L. Guo, B. Zhang, Polylogarithms and multiple zeta values
from free Rota-Baxter algebras, Sci. China Math., 53 (2010)
2239--2258.






\bibitem{LHB}X. Li, D. Hou, C. Bai, Rota-Baxter
operators on pre-Lie algebras, J. Nonlinear Math. Phys., 14 (2007)
269--289.



\bibitem{Rot}G.-C. Rota,  Baxter algebras and combinatorial identities I, II, Bull. Amer. Math. Soc., 75 (1969)
 325--329; ibid. 75 (1969) 330--334.

\bibitem{Sw}M. E. Sweedler,  Hopf algebras. Mathematics Lecture Note Series W. A. Benjamin, Inc., New York 1969.

\end{thebibliography}

\end{document}